\documentclass{amsart}[11pt]

\usepackage[left=3cm,right=3cm,top=2cm,bottom=2cm]{geometry} 
\usepackage{amsmath} 
\usepackage{mathtools}
\usepackage[utf8]{inputenc}
\usepackage[english]{babel}

\usepackage{amssymb,amsthm, amsmath, amsfonts, amsbsy}
\usepackage{hyperref}
\usepackage[all]{xy}
\usepackage{enumerate}
\usepackage[mathscr]{eucal}
\usepackage{bbm}

\theoremstyle{plain}
\newtheorem{theorem}{Theorem}[section]
\newtheorem{lemma}[theorem]{Lemma}
\newtheorem{proposition}[theorem]{Proposition}

\newtheorem{corollary}[theorem]{Corollary}

\numberwithin{equation}{section}

\theoremstyle{definition}

\newtheorem{definition}[theorem]{Definition}

\newtheorem{remark}[theorem]{Remark}

\DeclareMathOperator{\Aut}{Aut}

\DeclareMathOperator{\Res}{Res}
\DeclareMathOperator{\Ext}{Ext}
\DeclareMathOperator{\Mod}{-Mod}
\DeclareMathOperator{\module}{-mod}
\DeclareMathOperator{\inj}{-inj}
\DeclareMathOperator{\Ind}{Ind}

\DeclareMathOperator{\Hom}{Hom}
\DeclareMathOperator{\End}{End}
\DeclareMathOperator{\mor}{Mor}

\newcommand{\C}{{\mathscr{F}}}
\newcommand{\Cop}{{\mathscr{F}^{op}}}
\newcommand{\Cm}{{\mathscr{F}^m}}

\newcommand{\CmG}{{\mathscr{F}^m_G}}

\newcommand{\CS}{{\mathscr{F}^S}}
\newcommand{\CSG}{{\mathscr{F}^S_G}}
\newcommand{\CT}{{\mathscr{F}^{\neg S}}}
\newcommand{\CTG}{{\mathscr{F}^{\neg S}_G}}

\newcommand{\bft}{{\mathbf{t}}}
\newcommand{\bfn}{{\mathbf{n}}}
\newcommand{\bff}{{\mathbf{f}}}

\newcommand{\bfo}{{\mathbf{1}}}
\newcommand{\bfs}{{\mathbf{s}}}

\newcommand{\Sinduced}{{$S$-induced}}
\newcommand{\Stf}{{$S$-torsion free}}
\newcommand{\Shd}{{$S$-homological degree}}
\newcommand{\DS}{{\mathbf{D}_S}}

\newcommand{\KS}{{\mathbf{K}_S}}
\newcommand{\FI}{{\mathrm{FI}}}
\newcommand{\FIM}{{\mathrm{FI}^m}}
\newcommand{\Fs}{{F_\bfs}}
\newcommand{\tzero}{{t_0^S}}
\newcommand{\tone}{{t_1^S}}
\newcommand{\ti}{{t_i^S}}
\newcommand{\Hzero}{{H_0^S}}
\newcommand{\Hone}{{H_1^S}}
\newcommand{\Htwo}{{H_2^S}}
\newcommand{\Hi}{{H_i^S}}
\newcommand{\IS}{{\mathfrak{I}_S}}
\newcommand{\kCmI}{{k\Cm/\IS}}
\newcommand{\ShiftS}{{\mathbf{\Sigma}_S}}

\newcommand{\obj}{{\mathrm{Obj}}}
\newcommand{\Rs}{{\mathscr{R}_\bfs}}

\newcommand{\U}{{\mathcal{U}}}

\title{A classification of injective $\FIM$-modules}

\author{Duo Zeng}
\address{LCSM (Ministry of Education), School of Mathematics and Statistics, Hunan Normal University, Changsha, Hunan 410081, China.}
\email{zengduo@hunnu.edu.cn}

\thanks{The author is partially supported by the National Natural Science Foundation of China (Grant No. 11771135) of his advisor Liping Li. He would like to thank Prof. Li for leading him into this area, and for the numerous discussions and suggestions.}

\begin{document}

\begin{abstract}
In this paper we generalize a shift theorem, which plays a key role in studying representations of $\FIM$, the product category of the category of finite sets and injections, and classify finitely generated injective $\FIM$-modules over a field of characteristic 0.
\end{abstract}

\maketitle



\section{Introduction}
\subsection{Motivation}
The representation theory of infinite combinatorial categories has attracted much attention. It is mainly concerned with how an infinite category acts on a category of modules as they have close relations to (co)homological groups of topological spaces, geometric groups, and algebraic varieties. Among quite a few frequently concerned examples, the most important infinite combinatorial category is the category $\FI$ of finite sets and injections whose representation theoretic and homological properties are extensively studied;  see \cite{CEFN}.

Due to the importance of the category $\FI$, the structure of finitely generated injective $\FI$-module is of interest to many mathematicians. In \cite{SamSnowden}, Sam and Snowden firstly classified all injective $\FI$-modules over a field of characteristic 0, and proved that every finitely generated $\FI$-modules has finite injective dimension. In \cite{GanLi}, Gan and Li give another proof of this fact by introducing the coinduction functor for $\FI$-modules. These results give a deep homological explanation for the following crucial result established by Church, Ellenberg and Farb in \cite{CEF}: a sequence of representations of symmetric groups over a field of characteristic 0 encoded by an $\FI$-module is representation stable if and only if it is a finitely generated $\FI$-module.

One of the natural generalization of the category $\FI$ is the product category $\FIM$ whose representation theory has also been studied; see for instance \cite{Gadish, lr, LiYu}. However, the classification of the injective $\FIM$-modules is not covered in \cite{LiYu} and remains as an open problem at that time. In a recent work \cite{Zeng}, by extending the method used in \cite{GanLi}, the author showed the locally self-injective property of $\FIM$ over fields of characteristic $0$ and further found the external product of finitely generated injective $\FI$-modules being necessarily injective as $\FI^m$-module. In this paper, by generalizing certain concepts in \cite{LiYu} and utilizing an inductive method, the author successfully classifies all finitely generated injective $\FIM$-modules. Surprisingly, it turns out that the finitely generated indecomposable injective $\FIM$-modules are already found by the previous work \cite{Zeng}, i.e. they are exactly the external tensor product of $m$ many indecomposable injective $\FI$-modules.

\subsection{Main results}

Before describing the main results of this paper, let us introduce a few notations. Throughout this paper let $m$ be a positive integer and denote by $[m]$ the set $\{1,\ldots,m\}$. Let $\FI$ be the category of finite sets and injections. For brevity, we denote by $\C$ the full subcategory of $\FI$ consisting of objects $[n]$, $n \in \mathbb{N}$ and by $\Cm$ the product category of $m$ copies of $\C$. For any subset $S$ of $[m]$, we may form a product category $\CS$ of $\C$ indexed by $S$.

Let $k$ be a field of characteristic zero. For a locally small category $\mathcal{C}$, a \textit{representation} of $\mathcal{C}$ or a \textit{$\mathcal{C}$-module} over $k$ is a covariant functor from $\mathcal{C}$ to $k\Mod$, the category of vector spaces over $k$. We denote by $\mathcal{C} \Mod$ the category of all representations of $\mathcal{C}$ over $k$ and by $\mathcal{C} \module$ the category of finitely generated representations, which are quotients of direct sums of finitely many representable functors.

Let $\mathcal{C}$ and $\mathcal{D}$ be two locally small categories. There is a way to construct a $\mathcal{C} \times \mathcal{D}$-module from a pair of $\mathcal{C}$-modules and $\mathcal{D}$-module. Explicitly, given a $\mathcal{C}$-module $V_1$ and a $\mathcal{D}$-module $V_2$, we define their \textit{external tensor product} to be the $\mathcal{C} \times \mathcal{D}$-module, denoted by $V_1\boxtimes V_2$, such that
\[
    (V_1\boxtimes V_2)(c \times d) = V_1(c) \otimes_k V_2(d)
\]
where $c \in \obj(\mathcal{C})$ and $d \in \obj(\mathcal{D})$. For more details, see \cite[Definition 6.4]{Gadish}

Now we are ready to describe the main results of this paper. The first main result generalizes the shift theorem \cite[Proposition 4.10]{LiYu} by taking $S = [m]$, where $\Sigma_i$ is the $i$-th shift functor. For more details, see Section 2.

\begin{theorem} \label{theorem - applying_shift_makes_any_module_S-relative_projective}
Let $V$ be a finitely generated $\Cm$-module. Then there exists some positive integer $N$ such that $(\prod_{i \in S} \Sigma_i)^n V$ is $S$-semi-induced for $n \geqslant N$.
\end{theorem}

The classification of finitely generated injective $\C$-module was first accomplished by Sam and Snowden in \cite{SamSnowden}. Later, Gan and Li gave a new and independent proof for this result in \cite{GanLi}, utilizing properties of the coinduction functor which is right adjoint to the shift functor. Extending Gan and Li's method, we showed in \cite{Zeng} that finitely generated projective $\Cm$-module is injective, and furthermore, external tensor products of finitely generated injective $\C$-modules are injective $\Cm$-modules. The second main theorem strengthens this result.
\begin{theorem} \label{theorem - classification_of_f.g._injective_FIm-module}
    Any finitely generated indecomposable injective $\Cm$-module is isomorphic to $I_1 \boxtimes \ldots \boxtimes I_m$ where each $I_i$ is a finitely generated indecomposable injective $\C$-module for $i = 1, \ldots, m$.
\end{theorem}

Since we already know from \cite{SamSnowden} that an indecomposable injective $\C$-module is either an indecomposable projective $\C$-module or a finite dimensional indecomposable injective $\C$-module, which can be explicitly constructed, the above theorem actually gives a complete classification of finitely generated injective $\Cm$-modules.

\section{preliminaries}

In this section we give necessary notations, definitions, and some elementary results used throughout this paper. Since some results are generalizations of corresponded results described in \cite{LiYu} and can be established with the essentially same ideas or arguments, occasionally we omit detailed proofs and suggest the reader to see \cite{LiYu} for details.

\subsection{Some notations}

Recall that objects in $\Cm$ are of the form $\bfn = ([n_1], \, \ldots, \, [n_m])$. We denote by $\bfn + \bfn'$ the object in $\Cm$ whose $i$-th component is $[n_i+n'_i]$, and by $\bfo_i$ the object whose $i$-th component is the singleton set $[1]$ and all other components are empty sets. The \textit{degree} of an object $\bfn$, denoted by $\deg(\bfn)$, is defined to be the integer $\sum_i n_i$. For a morphism $\alpha: \bfn \to \bfn'$, we define $\deg(\alpha)$ to be the integer $\deg(\bfn') - \deg(\bfn)$. We also mention that there is a partial order $\preccurlyeq$ defined on $\obj(\Cm)$ by specifying $\bfn \preccurlyeq \bfn'$ if $\Cm(\bfn, \bfn') \neq \emptyset$.

Let $V$ be an $\Cm$-module. We denote the value of $V$ on an object $\bfn$ by $V(\bfn)$. For a morphism $\alpha: \bfn \to \bfn'$ in $\Cm$ and an element $v \in V(\bfn)$, we denote by $\alpha \cdot v$ the element $V(\alpha)(v) \in V(\bfn')$. Let $\Cm \Mod$ be the category of all $\Cm$-modules. It is well known that this category is abelian and has enough projective objects. In particular, for an object $\bfn$ in $\Cm$, the $k$-linearization of the representable functor $\Cm(\bfn, -)$ is a projective $\Cm$-module. We denote it by $M(\bfn)$, and we say that an $\Cm$-module is a \textit{free module} if it is isomorphic to a direct sum of $k$-linearizations of representable functors.

A key technical tool for studying representations of $\C$ is an endofunctor on $\C \Mod$, which is introduced in \cite{CEF} and called \textit{shift functor}. For convenience of the readers, we present here its definition and generalization. There is a self-embedding functor $\iota$ on $\C$ such that $\iota([n]) = [n+1]$, and for a morphism $f:[n] \to [t]$ in $\C$, $\iota(f)$ is a morphism from $[n+1]$ to $[t+1]$ with
\[
\iota(f)(x) =
\begin{cases}
1, & x = 1 \\
f(x-1)+1, & x \neq 1
\end{cases}
\]
for element $x \in [n+1]$. The \textit{shift functor} $\Sigma$ is defined to be the endofunctor on $\C\Mod$ sending an $\C$-module $V$ to the $\C$-module $V \circ \iota$. Furthermore, there is a natural transformation between the identity functor and the shift functor, so we get a natural homomorphism $V \to \Sigma V$, and hence obtain the kernel functor $K$ and the cokernel functor $D$, which is called the \textit{derivative functor} in \cite{CEF}.

The self-embedding functor and associated shift functor on $\C$-modules induce $m$ distinct self-embedding functors and shift functors on $\Cm$-modules. Explicitly, for  $i \in [m]$, the product category $\Cm$ can be viewed as product $\C^{ [m] \setminus \{i\} } \times \C$. The $i$\emph{-th self-embedding} functor is defined to be the endofunctor $\iota_i := \mathrm{Id} \times \iota$ on $\Cm$ where $\mathrm{Id}$ is the identity functor on the category $\C^{ [m] \setminus \{i\} }$. The $i$\emph{-th shift functor} $\Sigma_i$ is defined to be the endofunctor on $\Cm\Mod$ sending an $\Cm$-module $V$ to the $\Cm$-module $V \circ \iota_i$. There are also the $i$-th \textit{kernel functor} $K_i$ and the $i$-th \textit{derivative functor} $D_i$ defined on the category $\Cm \Mod$; for their definitions and elementary properties, see \cite[Section 2.2]{LiYu}. Remark that the kernel functor commutes with the shift functor, i.e. $K_i\Sigma_j \cong \Sigma_jK_i$ for all $i,j \in [m]$.

A main goal of this paper is to extend quite a few results in \cite{LiYu} from the full set $[m]$ to an arbitrary nonempty subset $S$ of $[m]$. For this purpose, we need to introduce a few constructions. Fix a subset $S$ of $[m]$, and denote $\neg S$ the complement subset $[m] \setminus S$. Then $\Cm$ is the product category of $\CS$ and $\CT$. Accordingly, an object $\bfn = ([n_1], \, \ldots, \, [n_m]) \in \obj(\Cm)$ can be written as a product $\bfs \times \bft$ for some object $\bfs \in \obj(\CS)$ and object $\bft \in \obj(\CT)$. We define the \textit{$S$-degree} of $\bfn$ to be the degree of $\bfs$ in the category $\CS$ and denote it by $\deg_S(\bfn)$. For brevity, we write $\deg_i$ for $\deg_{\{i\}}$ where $i \in [m]$. We denote by $\ShiftS$ the endofunctor on $\Cm \Mod$ which is the direct sum of $\Sigma_i$ for all $i \in S$, i.e. $\ShiftS = \bigoplus_{i\in S} \Sigma_i$. The endofunctors $\KS$ and $\DS$ on the category $\Cm \Mod$ are defined similarly.

In the proofs of main results we have to deal with the following categories which are generalizations of $\Cm$. Let $G$ be a finite group, and view it as a category with a single object. We define $\CmG$ to be the product category of $\Cm$ and $G$. Explicitly, objects of $\CmG$ coincide with objects of $\Cm$, and morphisms of $\CmG$ are ordered pairs $(\alpha,g)$ where $\alpha \in \mor(\Cm)$ and $g \in G$. Clearly, when $G$ is the trivial group, then $\CmG$ is precisely $\Cm$. Furthermore, we remark that many results from \cite{LiYu} are still valid for $\CmG$, so we will restate some of them in this paper without providing detailed proofs. In particular, functors $\Sigma_i$, $K_i$, and $D_i$ can be extended to the category $\CmG \Mod$ in a natural way and we keep the same notations. It is worthy to remark that the category $\CmG$ is also locally Noetherian over $k$; that is, submodules of finitely generated $\CmG$-modules are still finitely generated.

\subsection{$S$-Torsion theory}

In this subsection we introduce a torsion theory with respect to the nonempty subset $S \subseteq [m]$, and give a few elementary results.

\begin{definition} \label{definition - S-torsion}
Let $V$ be an $\CmG$-module and $\bfn$ an object in $\CmG$. An element $v \in V_{\bfn}$ is called \textit{$S$-torsion} if there exists a morphism $\alpha$ in $\CmG$ such that $\deg_S(\alpha) > 0$, $\deg_{\neg S}(\alpha) = 0$, and $\alpha \cdot v = 0$.
\end{definition}

Suppose that $\alpha: \bfn \to \bfn'$ is a morphism in the above definition. Then $\deg_S (\alpha) > 0$ means that there exists a certain $i \in S$ such that $n'_i > n$, and $\deg_{\neg S}(\alpha) = 0$ means that $n'_j = n_j$ for all $j \in \neg S$. Loosely speaking, this means that $\alpha$ is a morphism along the $S$-direction.

Let $V$ be a non-zero $\CmG$-module. It is clear that all $S$-torsion elements in $V$ form a submodule which is called the \textit{$S$-torsion part} of $V$ and denoted by $V_T^S$. The \textit{$S$-torsion free part} of $V$ is defined to be the quotient module $V/V_T^S$ and denoted by $V_F^S$. Then we have a short exact sequence $0 \to V_T^S \to V \to V_F^S \to 0$. We say that an $\CmG$-module is \textit{$S$-torsion} (resp., \textit{$S$-torsion free}) if and only if its torsion free part (resp., torsion part) is zero. We remark that when taking $S = [m]$, these definitions coincide with the ones introduced in the paper \cite{LiYu}.

There is another type of torsion submodule which is, in some sense, ``stronger" than the above one. Let $S$ be a subset of $[m]$, we denote by $V^S_{tor}$ the submodule
\[
V^S_{tor} := \bigcap_{i \in S} V_T^{ \{i\} }
\]
of $V$. More transparently, one has
\[
V^S_{tor} = \bigoplus_{\bfn \in \obj(\CmG)} \{ v \in V(\bfn) \mid \forall i \in S, \exists \alpha_i \in \mor(\CmG), \; \deg_i(\alpha_i) > 0, \deg_{[m] \setminus \{i\}}(\alpha_i) = 0, \text{ and }\alpha_i \cdot v = 0 \}.
\]

\begin{remark} \normalfont
Let $v$ be an element in $V(\bfn)$ for a certain object $\bfn$ in $\CmG$. Then $v$ is contained in $V^S_T$ if it eventually vanishes along the $i$-th direction for a certain $i \in S$, and $v$ is contained in $V^S_{tor}$ if it eventually vanishes along the $i$-th direction for all $i \in S$. Keeping in mind this intuition, it is easy to see that the quotient module $V^{[i-1]}_{tor} / V^{[i]}_{tor}$ is $\{i\}$-torsion free. Furthermore, if $V$ is finitely generated, then the submodule $V^{[m]}_{tor}$ is finite dimensional, a consequence of the locally Noetherian property of $\CmG$ over $k$.
\end{remark}

In the rest of this subsection we state a few elementary results on $S$-torsion theory, which have been established in \cite{LiYu} for the special case that $S = [m]$.

\begin{lemma} \label{lemma_properties_of_C_module}
Let $V$ be an $\CmG$-module. Then:
    \begin{enumerate}
    \item $V$ is \Stf~if and only if $\KS V = 0$, or equivalently, $K_iV = 0$ for all $i \in S$. \label{statement - relation_between_S-torsion-free_and_kernel_functor}
    \item If V is \Stf, then so is $\Sigma_i V$ for all $i \in S$. \label{statement - shift_functor_preserves_S-torsion-free}
    \item In a short exact sequence $0 \to U \to V \to W \to 0$, if both $U$ and $W$ are \Stf, so is $V$. \label{statement - S-torsion-free_is_closed_under_extension}
    \item \label{statement_short_sequence_is_exact} In a short exact sequence $0 \to U \to V \to W \to 0$, if $W$ is \Stf, then the sequence $0 \to \DS U \to \DS V \to \DS W \to 0$ is exact as well. In particular, $0 \to D_i U \to D_i V \to D_i W \to 0$ is exact for all $i \in S$.
    \end{enumerate}
\end{lemma}

\begin{proof}
    (\ref{statement - relation_between_S-torsion-free_and_kernel_functor}): If $\KS V \neq 0$, then there exists a certain $i \in S$ such that $K_i V \neq 0$. By the definition of $K_i$ and Definition \ref{definition - S-torsion}, the $\CmG$-module $V$ contains nonzero $S$-torsion elements, and hence is not \Stf. Conversely, if $V$ is not \Stf, we can find a morphism $\bff \in \CmG(\bfn, \bft)$ satisfying the condition in Definition \ref{definition - S-torsion}, and a nonzero element $v \in V(\bfn)$ such that $\bff \cdot v  = 0$. By a simple induction on the degree of morphisms one can assume that $\deg_S(\bff) = 1$; that is, there is a certain $i \in S$ such that $\deg_i(\bff) = 1$. Clearly, $v \in K_i V$, so $\KS V \neq 0$.

    (\ref{statement - shift_functor_preserves_S-torsion-free}) Assume that $\Sigma_iV$ is not $S$-torsion free; that is, there is some nonzero $v \in \Sigma_iV (\bfn) = V(\bfn + \bfo_i)$ for a certain object $\bfn$ such that $v$ is sent to $0$ by some morphism $\alpha$ satisfying the condition in Definition \ref{definition - S-torsion}. By the definition of the $i$-th shift functor $\Sigma_i$, $v$ is sent to $0$ by the morphism $\iota_i(\alpha)$, which also satisfies the condition specified in Definition \ref{definition - S-torsion}. Thus $V$ is not $S$-torsion free, which is a contradiction.

    (\ref{statement - S-torsion-free_is_closed_under_extension}): Applying the exact functor $\ShiftS$ one gets a commutative diagram where all vertical rows represent natural maps:
    \begin{equation*}
    \xymatrix{
    0 \ar[r] & U \ar[r] \ar[d] & V \ar[r] \ar[d] & W \ar[r] \ar[d] & 0 \\
    0 \ar[r] & \ShiftS U \ar[r] & \ShiftS V \ar[r] & \ShiftS W \ar[r] & 0.
    }
    \end{equation*}
    According to statement (\ref{statement - relation_between_S-torsion-free_and_kernel_functor}), the maps $U \to \ShiftS U$ and $W \to \ShiftS W$ are injective, so is the map $V \to \ShiftS V$ by the snake Lemma. Therefore, by the first statement, $V$ is $S$-torsion free as well.

    (\ref{statement_short_sequence_is_exact}): Follows from the above commutative diagram and the snake Lemma.
\end{proof}

We present some useful properties of the functors $\Sigma_i$ and $D_i$.
\begin{lemma}[{\cite[Lemma 2.3]{LiYu}}] \label{lemma - properties_of_shift_and_derivative_functors}
    For $i,j \in [m]$, one has:
    \begin{enumerate}
        \item $\Sigma_i M(\bfn) \cong M(\bfn) \oplus M(\bfn - \bfo_i)^{\oplus n_i}$.
        \item $D_iM(\bfn) \cong M(\bfn - \bfo_i)^{\oplus n_i}$.
        \item $\Sigma_i \circ \Sigma_j = \Sigma_j \circ \Sigma_i$.
        \item $\Sigma_i \circ D_j = D_j \circ \Sigma_i$.
    \end{enumerate}
\end{lemma}

In the situation that $V$ is finitely generated, one can apply the shift functors to eliminate its torsion part.

\begin{lemma}[{\cite[Lemma 4.8]{LiYu}}] \label{lemma - existence_of_a_sufficient_large_n_to_ensure_zero_kernel}
Let $V$ be a finitely generated $\CmG$-module over a field of characteristic 0 and $i \in [m]$. Then $K_i \Sigma_i^n V = 0$ for $n$ sufficiently large.
\end{lemma}

\subsection{Slices and $S$-Homology groups}

In this subsection we introduce $S$-homology groups and $S$-homological degrees. As before, let $S$ be a nonempty subset of $[m]$. The category $\CmG$ is a product of the category $\CS$ and the category $\CTG$. We will see later that the study about the injectivity of any finitely generated $\CmG$-module can be reduced to that of modules over the subcategory $\Rs = \Aut(\bfs) \times \CTG$ of $\CmG$ for certain object $\bfs \in \obj(\CS)$. In view of this observation, it is necessary to introduce the functor $\overline{M}(\bfs) \otimes_\Rs -$ that produces an $\CmG$-module from an $\Rs$-module.
\begin{definition} \label{definition - the exact tensor functor on R-modules}
For object $\bfs \in \obj(\CS)$, let $\Rs := \Aut(\bfs) \times \CTG$ be the subcategory of $\CmG$. There is a functor
\[
\overline{M}(\bfs) \otimes_\Rs -: \Rs \Mod \to \CmG \Mod
\]
where $\overline{M}(\bfs) := M(\bfs) \boxtimes k\CTG$ is a $(\CmG, \Rs)$-bimodule, $M(\bfs)$ is free as an $\CS$-module, and the category algebra $k\CTG$ is a $(\CTG, \CTG)$-bimodule. We will denote this functor by $\Fs$ throughout this paper.
\end{definition}

In the above definition, the right $\Rs$-module structure of $\overline{M}(\bfs)$ follows from the right $k\Aut(\bfs)$-module structure of $M(\bfs)$ together with the right $\CTG$-module structure of $k\CTG$. One has that $\overline{M}(\bfs)$ is free as right $\Rs$-module since the $\CS$-module $M(\bfs)$ is free as right $k\Aut(\bfs)$-module. Therefore the functor $\Fs$ is exact and preserves projectives.

\begin{definition} \label{definition - slice}
For an object $\bfs \in \obj(\C^S)$, we define the \textit{slice} of $V$ on $\bfs$ to be the $\CmG$-module (which is also an $\Rs$-module)
\[
    V[[\bfs]] = \bigoplus_{\bft \in \obj(\C^{\neg S}_G)} V(\bfs \times \bft).
\]
\end{definition}

We present here an observation. For a finitely generated $\CmG$-module $V$, the slice $V[[\bfs]]$ may not be a submodule of $V$. However, for the homology module $\Hzero(V)$ that is defined in the next paragraph, the slice $\Hzero(V)[[\bfs]]$ is a direct summand of it and it is a direct sum of finitely many such summands. Also, for $i \in S$, the $\CmG$-module $K_i V$ is a direct sum of its slices on objects in the category ${\C}^{\{i\}}$.

Now we are ready to define $S$-homology groups for $\CmG$-modules. Let $\mathfrak{I}_S$ be the free $k$-module spanned by all morphisms $\alpha$ with $\deg_S(\alpha) > 0$. It is easy to see that $\mathfrak{I}_S$ is a two-sided ideal of the category algebra $k\CmG$. For $V \in \CmG \Mod$, define
\[
\Hzero(V) = k\CmG/\mathfrak{I}_S \otimes_{k\CmG} V.
\]
 Note that the module $\Hzero(V)$ is isomorphic to the quotient $V/\IS V$. Further, for $i > 0$, we define the $i$-th $S$-\textit{homology functor} to be the $i$-th left derived functor of $\Hzero$ and denote it by $\Hi$. The $\CmG$-module $\Hi(V)$ is called the $i$-th $S$-\textit{homology group} of $V$.

For $i \geqslant 0$, we define the \textit{$i$-th \Shd} to be the integer
\[
\ti(V) = \sup \{\deg(\bfs) \mid \Hi(V)[[\bfs]] \neq 0, \; \bfs \in \obj(\CS)\}.
\]
If the set on the right hand side is empty, we set $\ti(V) = -1$.

We collect some elementary results about $S$-homology groups in the following lemmas.

\begin{lemma} \label{lemma - inequality of S-homolgical degree in a long exact sequence}
For a short exact sequence of $\CmG$-module
\begin{equation*}
0 \to V' \to V \to V'' \to 0,
\end{equation*}
we have
\begin{enumerate}
    \item $t_{i+1}^S(V'') \leqslant max\{t_{i+1}^S(V), \; \ti(V')\}$
    \item $\ti(V) \leqslant max\{\ti(V'), \; \ti(V'')\}$
    \item $\ti(V') \leqslant max\{\ti(V), \; t_{i + 1}^S(V'')\}$
\end{enumerate}
for $i \geqslant 0$.
\end{lemma}
\begin{proof}
The conclusion follows from the long exact sequence
\[
    \ldots \to \Hi(V') \to \Hi(V) \to \Hi(V'') \to H_{i-1}^S(V') \to \ldots \to \Hzero(V') \to \Hzero(V) \to \Hzero(V'') \to 0.
\]
\end{proof}

\begin{lemma} \label{lemma_derivative_functor_decreases_generating_degree}
If $V \in \CmG \Mod$ is nonzero, then $\tzero (\DS V) = \tzero(V)-1$.
\end{lemma}

\begin{proof}
    As explained in \cite[Lemma 1.5]{two_homological}, we only need to deal with the case that $\tzero(V)$ is finite. If $\tzero(V) = 0$, then an argument similar to the following one will yield $\DS P = 0$ hence $\DS V = 0$, as required.  Now we assume that $\tzero(V) > 0$. Since $V$ is finitely generated, we may find a surjection $P \to V$ where $P$ is a finitely generated projective $\CmG$-modules with $\tzero(P) = \tzero(V)$. Since the functor $\DS$ is right exact, the map $\DS P \to \DS V$ is surjective. By Lemma \ref{lemma - properties_of_shift_and_derivative_functors}, we have that $\tzero(\DS P) = \tzero(P) - 1$. Then one can deduce that $\tzero(\DS V) \leqslant \tzero(\DS P) = \tzero(P) - 1 = \tzero(V) - 1$.

    On the other hand, there is an object $\bfs \in \obj(\CS)$ such that $\deg(\bfs) = \tzero(V) \geqslant 1$ and $\Hzero(V)[[\bfs]] \neq 0$. Now let $V'$ be the submodule of $V$ generated by $V[[\bft]]$ with $\bfs \npreceq \bft$ and $\bft \in \obj(\CS)$. Then $V[[\bfs]] \not \subseteq V'$ since otherwise one should have $\Hzero(V)[[\bfs]] = 0$. The surjection $V \to V/V' \to 0$ induces a surjection $\DS V \to \DS(V/V') \to 0$, and hence $\tzero(\DS V) \geqslant \tzero(\DS (V/V'))$. Note that $(V/V')[[\bft]] \neq 0$ implies $\bfs \preccurlyeq \bft$ for $\bft \in \obj(\CS)$. Thus, for $i \in S$, we have that $\Sigma_i(V/V')[[\bft]] \neq 0$ implies $\bft \succcurlyeq \bfs - \bfo_i$ and that $D_i(V/V')[[\bft]] \neq 0$ implies $\bft \succcurlyeq \bfs - \bfo_i$. As a result, we have $\tzero(D_i(V/V')) \geqslant \deg(\bfs) - 1 = \tzero(V) - 1$ and so is $\tzero(\DS(V/V'))$. Combining the two inequalities, one obtains that $\tzero(\DS(V)) \geqslant \tzero(\DS(V/V')) \geqslant \tzero(V) - 1$, as required.
\end{proof}

\section{$S$-induced modules and $S$-semi-induced modules}

In this section we consider $S$-induced modules and $S$-semi-induced modules, which are generalizations of induced modules and semi-induced modules considered in \cite{LiYu}. \textbf{All modules considered in the rest of this paper are finitely generated unless otherwise specified.}

\begin{definition}
A finitely generated $\CmG$-module $V$ is said to be \textit{\Sinduced} if $V \cong \Fs (W)$ for certain object $\bfs$ in $\CS$ and some $\Rs$-module $W$.
\end{definition}

The \Sinduced~module $V$ has the following universal property. Let $W$ be an $\Rs$-module (also viewed as an $\CmG$-module) and $N$ an $\CmG$-module. Then a homomorphism $W \to N$ as $\CmG$-module uniquely extends to a homomorphism $\Fs(W) \to N$ as $\CmG$-module.

\begin{definition} \label{definition - S-relative projective}
An $\CmG$-module $V$ is said to be \textit{$S$-semi-induced} if it has a finite filtration
\[
0 = V^0 \subseteq V^1 \subseteq \ldots \subseteq V^n = V
\]
such that for each $i$, the quotient module $V^{i+1}/V^i$ is \Sinduced.
\end{definition}

We remark that when $S = [m]$, $S$-semi-induced modules coincide with relative projective modules defined in \cite{LiYu}. Further, over a field of characteristic zero, relative projective modules coincide with projective modules.

\begin{lemma} \label{lemma_properties_of_module_genrated_on_one_slice}
Let $V$ be an $\CmG$-module generated by its slice $V[[\bfs]]$ for a certain object $\bfs \in \obj(\C^S)$. One has:
    \begin{enumerate}
    \item \label{statement - equivalent condition for basic relative projective modules} The following are equivalent:
        \begin{itemize}
        \item $V$ is an \Sinduced~module;
        \item $\Hi(V) = 0$ for all $i \geqslant 1$;
        \item $\Hone(V) = 0$.
        \end{itemize}
    \item If $V$ is $S$-induced, then it is \Stf. \label{statement - S-relative_projective_implies_S-torsion_free}
    \item If $\tone(V) \leqslant \tzero(V)$, then $V$ is \Sinduced. \label{statement - t1 < t0_implies_S-acyclic}
    \item If $V$ is \Sinduced, then $\Sigma_i V$ is $S$-semi-induced and $D_i V$ is \Sinduced~for all $i \in S$. \label{statement - shift_and_derivative_functors_on_S-acyclic_module}
    \end{enumerate}
\end{lemma}

\begin{proof}
(1): Suppose that $V$ is an \Sinduced~module. Then one has $V \cong \overline{M}(\bfs) \otimes_{\Rs} V[[\bfs]]$ where $\Rs$ is defined as in Definition \ref{definition - the exact tensor functor on R-modules}. Take a projective presentation $0 \to W \to P \to V[[\bfs]] \to 0$ of the $\Rs$-module $V[[\bfs]]$. Applying the functor $\overline{M}(\bfs) \otimes_{\Rs} -$, we get a projective presentation of the $\CmG$-module $V$ as follows
\begin{equation*}
0 \to \overline{M}(\bfs) \otimes_{\Rs} W \to \overline{M}(\bfs) \otimes_{\Rs} P \to (\overline{M}(\bfs) \otimes_{\Rs}  V[[\bfs]]) \cong V \to 0.
\end{equation*}
Applying the functor $\kCmI \otimes_{k \Cm} -$ we recover the original short exact sequence. That is, $\Hone(V) = 0$. Replacing $V$ by $V' = \overline{M}(\bfs) \otimes_{\Rs} W$ we deduce that $\Htwo(V) = 0$. Recursively, for every $i \geqslant 1$, one gets $\Hi(V) = 0$.

Conversely, suppose that $\Hone(V) = 0$. Since $V$ is generated by $V[[\bfs]]$, there is a short exact sequence of $\CmG$-module
\begin{equation} \label{sequence - short exact sequence}
0 \to K \to N \to V \to 0.
\end{equation}
where $N = \overline{M}(\bfs) \otimes_{\Rs} V[[\bfs]]$. The long exact sequence of $\Cm$-modules
\begin{equation} \label{sequence - long exact sequence}
\ldots \to \Hone(V) = 0 \to \Hzero(K) \to \Hzero (N) = V[[\bfs]] \to \Hzero(V) = V[[\bfs]] \to 0
\end{equation}
implies that $\Hzero(K) = 0$. That is, $K = 0$, and hence $V \cong N$ is \Sinduced.

(2): By definition, we have that $V \cong \overline{M}(\bfs) \otimes_\Rs V[[\bfs]]$. Let $\bfs'$ and $\bfs''$ be objects in $\CS$ with $\bfs \prec \bfs' \preccurlyeq \bfs''$, $\bft$ an object in $\C_G^{\neg S}$, and $\bff: \bfs' \times \bft \to \bfs'' \times \bft$ a morphism in $\CmG$. By an argument similar to the proof of \cite[Lemma 4.2(2)]{LiYu}, we have that $\bff \cdot v \neq 0$ for non-zero element $v \in V(\bfs' \times \bft)$. Therefore $V$ is \Stf.

(3): Again, consider exact sequences \eqref{sequence - short exact sequence} and \eqref{sequence - long exact sequence}. Since $\tone(V) \leqslant \tzero(V)$, by Lemma \ref{lemma - inequality of S-homolgical degree in a long exact sequence} we know that $\tzero(K) \leqslant \max \{\tone(V), \, \tzero(N) = \tzero(V)\} = \tzero(V) = \deg(\bfs)$, which means that $\Hzero (K)[[\bft]] \neq 0$ only if $\deg(\bft) \leqslant \deg(\bfs)$ for $\bft \in \obj(\CS)$. The sequence \eqref{sequence - short exact sequence} yields a short exact sequence of $\Rs$-modules
\begin{equation*}
    0 \to K[[\bfs]] \to N[[\bfs]] \to V[[\bfs]] \to 0.
\end{equation*}
So $K[[\bfs]] = 0$ since $N[[\bfs]] = V[[\bfs]]$. We have that $N[[\bfs']] \neq 0$ only if $\bfs' \succcurlyeq \bfs$ by our construction of $N$, so is the submodule $K$ of $N$. Putting the established results together we obtain that $\Hzero(K) = 0$, so $K = 0$. Therefore, $V \cong N$ is \Sinduced.

(4): As shown in the proof of statement (1), there is a short exact sequence $0 \to K \to P \to V \to 0$ such that $P$ is a projective $\CmG$-module generated by $P[[\bfs]]$. By the previous arguments we know that all terms in this sequence are \Sinduced~modules generated by their slice on the object $\bfs$, and hence are \Stf.
By Statement (\ref{statement_short_sequence_is_exact}) of Lemma \ref{lemma_properties_of_C_module}, for each $i \in S$, we get a short exact sequence $0 \to D_iK \to D_iP \to D_i V \to 0$. Since $D_i V = 0$ whenever $s_i$, the $i$-th component of $\bfs$, is zero, without loss of generality we assume that $s_i > 0$. Then $D_i V$ is generated by its slice on the object $\bfs - \bfo_i$ since so is $D_iP$. Replacing $V$ by $K$ one knows that $D_i K$ is also generated by its slice on $\bfs - \bfo_i$. The module $D_i P$ being projective implies that $\Hone(D_i P) = 0$, the long exact sequence of homology groups induced by $0 \to D_i K \to D_i P \to D_i V \to 0$ tells us that $\tzero(D_i V) = \tzero(D_i P) \geqslant \tzero (D_iK) \geqslant \tone(D_iV)$. By statement (\ref{statement - t1 < t0_implies_S-acyclic}), $D_iV$ is \Sinduced. But $D_i V \cong \Sigma_i V / V$, so $\Sigma_i V$ is $S$-semi-induced since $V$ is \Sinduced.
\end{proof}

In the following proposition we describe two homological characterizations of $S$-semi-induced modules.

\begin{proposition} \label{proposition - relations_between_S-relative_projective_and_H_1}
For $V \in \CmG \module$, the following are equivalent:
\begin{enumerate}
\item $V$ is an $S$-semi-induced module; \label{statement - is S-relative projective}
\item $\Hi(V) = 0$ for all $i \geqslant 1$; \label{statement - H_i = 0 for i > 0}
\item $\Hone(V) = 0$. \label{statement - H_1 = 0}
\end{enumerate}
\end{proposition}

\begin{proof}
(\ref{statement - is S-relative projective}) $\Rightarrow$ (\ref{statement - H_i = 0 for i > 0}): Let
\[
0 = V^0 \subseteq V^1 \subseteq \ldots \subseteq V^n = V
\]
be the filtration described in Definition \ref{definition - S-relative projective}. We prove by making induction on the superscripts of modules in this filtration. By statement (\ref{statement - equivalent condition for basic relative projective modules}) of Lemma \ref{lemma_properties_of_module_genrated_on_one_slice} we have that $\Hi(V^1) = 0$. Assume that $\Hi(V^k) = 0$. We have a short exact sequence
\[
0 \to V^k \to V^{k+1} \to V^{k+1}/V^k \to 0
\]
where $V^{k+1}/V^k$ is \Sinduced. Again, by statement (\ref{statement - equivalent condition for basic relative projective modules}) of Lemma \ref{lemma_properties_of_module_genrated_on_one_slice}, we have that $\Hi(V^{k+1}/V^k) = 0$ for $i \geqslant 1$. The long exact sequence
\[
\ldots \to \Hi(V^k) = 0 \to \Hi(V^{k+1}) \to \Hi(V^{k+1}/V^k) = 0 \to \ldots
\]
tells us that $\Hi(V^{k+1}) = 0$. The conclusion then follows by induction.

(\ref{statement - H_i = 0 for i > 0}) $\Rightarrow$ (\ref{statement - H_1 = 0}): trivial.

(\ref{statement - H_1 = 0}) $\Rightarrow$ (\ref{statement - is S-relative projective}): On the one hand, the comment following Definition \ref{definition - slice} says that the $\CmG$-module $\Hzero(V)$ is a direct sum of $\CmG$-modules $W_j$, where $W_j = \Hzero(V)[[\bfs_j]]$ is an $\mathscr{R}_{\bfs_j}$-module (also an $\CmG$-module) for  a set of objects $\bfs_j$ in $\CS$ and $1 \leqslant j \leqslant l$ for some positive integer $l$. In other words, $\Hzero(V) \cong \bigoplus_{j = 1}^l W_j$. On the other hand, we obtain a short exact sequence of $\CmG$-modules
\[
0 \to K \to N \to V \to 0.
\]
where $N := \bigoplus_{j = 1}^l F_{\bfs_j}(W_j)$ and $K$ is the kernel of the natural surjection $N \to V$. Since $\Hone(V) = 0$ and $\Hzero(N) = \bigoplus_{j = 1}^l W_j$, we have the long exact sequence
\[
\ldots \to \Hone(V) = 0 \to \Hzero(K) \to \bigoplus_{j = 1}^l W_j \to \Hzero(V) \to 0
\]
Since $\Hzero(V) \cong \bigoplus_{j = 1}^l W_j$, the kernel $\Hzero(K)$ is zero hence $K = 0$. Therefore $V \cong N$ is $S$-semi-induced.
\end{proof}

\begin{corollary} \label{corollary - properties_of_relative_projective_modules}
If $V$ is $S$-semi-induced, then it is \Stf, and $\Sigma_i V$ and $D_i V$ are $S$-semi-induced as well for all $i \in S$.
\end{corollary}

\begin{proof}
Suppose that $V$ admits a filtration
\[
0 = V^0 \subseteq V^1 \subseteq \ldots \subseteq V^n = V
\]
with each factor \Sinduced. The first part of the statement follows from statement (\ref{statement - S-relative_projective_implies_S-torsion_free}) of Lemma \ref{lemma_properties_of_module_genrated_on_one_slice} and the fact that \Stf~modules are closed under extension. Now we prove the second part of the statement. For $i \in S$, applying the exact functor $\Sigma_i$ to this filtration we get a filtration of $\Sigma_i V$ with factor
\[
\Sigma_i V^k /\Sigma_i V^{k-1} \cong \Sigma_i(V^k/V^{k-1})
\]
which is $S$-semi-induced by statement (\ref{statement - shift_and_derivative_functors_on_S-acyclic_module}) of Lemma \ref{lemma_properties_of_module_genrated_on_one_slice}. Therefore $\Sigma_i V$ is $S$-semi-induced. A similar argument together with statement (\ref{statement_short_sequence_is_exact}) of Lemma \ref{lemma_properties_of_C_module} shows that $D_i V$ is $S$-semi-induced.
\end{proof}

The following lemma is crucial for us to prove the first main result of this paper.

\begin{lemma} \label{lemma - derivative_functor_and_S-relative_projective}
Let $V$ be an \Stf~$\CmG$-module. If $\DS V$ is $S$-semi-induced, so is $V$.
\end{lemma}

\begin{proof}
The conclusion holds for $V = 0$ trivially, so we may assume that $V \neq 0$. Since $V$ is finitely generated, the set
\[
O= \{ \bfs \in \obj(\CS) \mid \Hzero(V)[[\bfs]] \neq 0 \}
\]
is finite. We prove by an induction on the cardinality of $O$. The conclusion holds clearly if $|O| = 0$. For $|O| \geqslant 1$, we choose an object $\bfs \in O$ such that $\deg(\bfs)$ is maximal; that is, $\deg(\bfs) = \tzero(V)$ (of course, this choice might not be unique). Let $V'$ be the submodule of $V$ generated by its slices on objects in $O \setminus \{ \bfs \}$. Then $V'' = V / V'$ is nonzero and is generated by $V''[[\bfs]]$.

Consider the short exact sequence $0 \to V' \to V \to V'' \to 0$. We claim that $V''$ is an \Sinduced~$\CmG$-module. To see this, from the long exact sequence of homology groups one has
\begin{equation*}
\tone(V'') \leqslant \max \{ \tzero(V'), \, \tone(V) \} \leqslant \max \{ \tzero(V), \, \tone(V) \},
\end{equation*}
where the second inequality follows from $\tzero(V') \leqslant \tzero(V)$ by our construction of $V'$. Furthermore, let $0 \to W \to P \to V \to 0$ be a short exact sequence of $\CmG$-modules such that $P$ is a free $\CmG$-module satisfying $\tzero(V) = \tzero(P)$. Applying $\DS$ we get another short exact sequence $0 \to \DS W \to \DS P \to \DS V \to 0$ such that $\DS P$ is also free and satisfies $\tzero(\DS V) = \tzero (\DS P)$ by Lemma \ref{lemma_derivative_functor_decreases_generating_degree}. Then one has
\begin{equation*}
\tone(V) \leqslant \tzero(W) = \tzero(\DS W) + 1 \leqslant \max \{ \tone(\DS V), \, \tzero(\DS V) \} + 1 \leqslant \tzero(\DS V) + 1 = \tzero(V)
\end{equation*}
where the third inequality follows from $\Hone(\DS V) = 0$ by Proposition \ref{proposition - relations_between_S-relative_projective_and_H_1}. Putting the above two inequalities together we conclude that $\tone(V'') \leqslant \tzero(V) = \tzero(V'')$. By statement (\ref{statement - t1 < t0_implies_S-acyclic}) of Lemma \ref{lemma_properties_of_module_genrated_on_one_slice}, $V''$ is \Sinduced~as claimed.

The conclusion follows after we show that $V'$ is $S$-semi-induced. By the induction hypothesis, it suffices to show that $\DS V'$ is $S$-semi-induced. Since $V''$ is \Sinduced, so is $\DS V''$ by statement (\ref{statement - shift_and_derivative_functors_on_S-acyclic_module}) of Lemma \ref{lemma_properties_of_module_genrated_on_one_slice} and it is \Stf~by statement (\ref{statement - S-relative_projective_implies_S-torsion_free}) of Lemma \ref{lemma_properties_of_module_genrated_on_one_slice}. By statement (\ref{statement_short_sequence_is_exact}) of Lemma \ref{lemma_properties_of_C_module}, we get a short exact sequence $0 \to \DS V' \to \DS V \to \DS V'' \to 0$. The long exact sequence of homology groups tells us that $\Hone(\DS V') = 0$, so $\DS V'$ is $S$-semi-induced by Proposition \ref{proposition - relations_between_S-relative_projective_and_H_1}.
\end{proof}

Now we are ready to prove the first theorem in the Introduction of this paper.

\begin{proof}[A proof of Theorem \ref{theorem - applying_shift_makes_any_module_S-relative_projective}]
By Lemma \ref{lemma - existence_of_a_sufficient_large_n_to_ensure_zero_kernel}, there exists a positive integer $c$ such that $\KS(\prod_{i \in S} \Sigma_i)^c V = 0$. Therefore $(\prod_{i \in S} \Sigma_i)^c V$ is \Stf~by statement (\ref{statement - relation_between_S-torsion-free_and_kernel_functor}) of Lemma \ref{lemma_properties_of_C_module}. Now we make induction on $\tzero(V)$. Assume that the conclusion holds for any $W \in \CmG \module$ with $\tzero(W) < \tzero(V)$. By Lemma \ref{lemma_derivative_functor_decreases_generating_degree}, we have that $\tzero(\DS V) < \tzero(V)$. By the induction hypothesis, there exists some integer $l > 0$ such that $(\prod_{i \in S} \Sigma_i)^l \DS V$ is $S$-semi-induced. Set $n := \max\{c,l\}$. By Corollary \ref{corollary - properties_of_relative_projective_modules} and the statement (4) of Lemma \ref{lemma - properties_of_shift_and_derivative_functors}, we have that
\[
\DS (\prod_{i \in S} \Sigma_i)^n V \cong (\prod_{i \in S} \Sigma_i)^n \DS V = (\prod_{i \in S} \Sigma_i)^{n-l} (\prod_{i \in S} \Sigma_i)^l \DS V
\]
is $S$-semi-induced. By statement (\ref{statement - shift_functor_preserves_S-torsion-free}) of Lemma \ref{lemma_properties_of_C_module}, we have that
\[
(\prod_{i \in S} \Sigma_i)^n V = (\prod_{i \in S} \Sigma_i)^{n - c} (\prod_{i \in S} \Sigma_i)^c V
\]
is \Stf. By Lemma \ref{lemma - derivative_functor_and_S-relative_projective}, we have that $(\prod_{i \in S} \Sigma_i)^n V$ is $S$-semi-induced.
\end{proof}

An immediate corollary is:

\begin{corollary} \label{proposition - S-torsion_free_module_can_be_embedded_into_S-relative_projective_module}
Let $V$ be a finitely generated $\CmG$-module. If $V$ is \Stf~then $V$ can be embedded into some finitely generated $S$-semi-induced $\CmG$-module.
\end{corollary}

\begin{proof}
This follows from Theorem \ref{theorem - applying_shift_makes_any_module_S-relative_projective}, statements (\ref{statement - relation_between_S-torsion-free_and_kernel_functor}) and (\ref{statement - shift_functor_preserves_S-torsion-free}) of Lemma \ref{lemma_properties_of_C_module} and the fact that the functor $\Sigma_i$ preserves finitely generated modules. Explicitly, there is an injective homomorphism $V \to (\prod_{i \in S} \Sigma_i)^n V$.
\end{proof}

\section{A Classification of Indecomposable Injective $\Cm$-Modules}

In this section we classify all indecomposable injective $\Cm$-modules. For this purpose, we firstly construct a class of modules that finitely cogenerates the category $\CmG \module$ and show that they are injective. Consequently, any finitely generated injective $\CmG$-module is isomorphic to a direct summand of a finite direct sum of modules in this class.

Let us introduce a few necessary notions. Let $\mathcal{C}$ be a small category. We denote by $\mathcal{C} \inj$ the category of all finitely generated injective $\mathcal{C}$-modules. For a $\mathcal{C}$-module $V$ and a class $\mathcal U$ of $\mathcal{C}$-modules, we say that $V$ is \textit{finitely cogenerated by} $\U$ if there exists a finite set $X$ and a map $f:X \to \U$ such that the $\mathcal{C}$-module $V$ can be embedded into the $\mathcal{C}$-module $\oplus_{x\in X} f(x)$.

As before, let $S$ be a nonempty subset of $[m]$. We denote by $\mathcal U^S_G$ the class of $\CSG$-modules
\[
    \mathcal U^S_G := \{ (\boxtimes_{i \in S} I_i) \boxtimes kG \mid  I_i \in \C \inj\}.
\]
We usually omit the subscript $G$ of $\U^S_G$ when there is no ambiguity. For brevity, we denote by $\U^n$ the class $\U^{[n]}$ for $1 \leqslant n \leqslant m$.

\begin{definition}
Let $\jmath: \Cm \to \CmG, \alpha \mapsto (\alpha, 1)$ be the embedding functor where $\alpha \in \mor(\Cm)$. It induces a pair $(\Ind, \Res)$ of functors
\[
\Res: \CmG \Mod \to \Cm \Mod; \quad \Ind: \Cm \Mod \to \CmG \Mod
\]
as follows: The functor $\Res$ sends an $\CmG$-module $W$ to the $\Cm$-module $W \circ \jmath$, and the functor $\Ind$ sends an $\Cm$-module $V$ to the $\CmG$-module $V \boxtimes kG \cong k\CmG \otimes_{k\Cm} V$ and sends an $\Cm$-morphism $\varphi : V \to V'$ to $\CmG$-morphism $\varphi \boxtimes Id : V \boxtimes kG \to V' \boxtimes kG$.
\end{definition}

We will show in Lemma \ref{lemma - adjuncation_of_Ind_and_Res} that $(\Ind, \Res)$ is an adjoint pair. Moreover, the functors $\Res$ and $\Ind$ are both exact functors, so $\Ind$ preserves projectives and $\Res$ preserves injectives.

\begin{lemma} \label{lemma - adjuncation_of_Ind_and_Res}
    The functor $\Res$ is right adjoint to $\Ind$.
\end{lemma}

\begin{proof}
We prove by constructing the adjunction directly. For an $\Cm$-module $V$ and an $\CmG$-module $W$, let
\[
\theta_{VW}: \Hom_{\CmG}(\Ind(V), W) \to \Hom_{\Cm}(V, \Res(W))
\]
be the map such that for an $\CmG$-module homomorphism $\varphi : V \boxtimes kG \to W$ we have that
\[
\theta_{VW}(\varphi)(v) = \varphi(v \otimes_k 1_G)
\]
for $v \in V$. Conversely, one can define a map
\[
\theta_{VW}^{-1}: \Hom_{\Cm}(V, \Res(W)) \to \Hom_{\CmG}(\Ind(V), W)
\]
such that for an $\Cm$-module homomorphism $f: V \to W$,
\[
\theta_{VW}^{-1}(f)(v \otimes_k g) = (e_\bfn, g) \cdot f(v)
\]
for $v \in V(\bfn)$, $g \in G$, and $e_\bfn$ is the identity morphism on the object $\bfn \in \obj(\Cm)$. It is routine to check that the homomorphisms $\theta_{VW}(\varphi)$ and $\theta_{VW}^{-1}(f)$ are well-defined, that the map $\theta_{VW}$ is a bijection with inverse $\theta_{VW}^{-1}$, and that $\theta: \Hom_{\CmG}(\Ind(-),-) \to \Hom_{\Cm}(-,\Res(-))$ is a natural equivalence.
\end{proof}

The following lemma says that the functor $\Res$ preserves finitely generated modules.

\begin{lemma} \label{lemma - restriction_functor_preserves_f.g._modules}
    An $\CmG$-module $V$ is finitely generated if and only if $\Res(V)$ is finitely generated as $\Cm$-module.
\end{lemma}
\begin{proof}
    The if part is trivial. For the only if part, suppose that the $\CmG$-module $V$ has a finite set $X$ of generators. Then the set
    \[
        X' := \{(1,g) \cdot x \mid g \in G, x \in X\}
    \]
    is a finite set since $G$ is finite. Moreover, $X'$ is a set of generators of the $\Cm$-module $\Res(V)$ since any morphism $(\alpha,g)$ in the category $\CmG$ can be decomposed as $(\alpha,1) \circ (1,g)$. Therefore the $\Cm$-module $\Res(V)$ is finitely generated.
\end{proof}

As we mentioned before, the category $\CmG$ is locally Noetherian over $k$. Here we give a new proof using the functors $\Ind$ and $\Res$.

\begin{lemma} \label{lemma - locally_Noetherian_property_of_CmG}
    The category $\CmG$ is locally Noetherian over a commutative Noetherian ring. That is, any $\CmG$-submodule of finitely generated $\CmG$-module is still finitely generated.
\end{lemma}
\begin{proof}

Suppose that $U$ is an $\CmG$-submodule of some finitely generated $\CmG$-module $V$. Since the functor $\Res$ is exact, $\Res(U)$ is a submodule of $\Res(V)$ as $\Cm$-module. By Lemma \ref{lemma - restriction_functor_preserves_f.g._modules}, $\Res(V)$ is finitely generated. By the Noetherian property of $\Cm$, see \cite[Theorem 1.1]{LiYu}, $\Res(U)$ is finitely generated. Again, by Lemma \ref{lemma - restriction_functor_preserves_f.g._modules}, the $\CmG$-module $U$ is finitely generated as desired.
\end{proof}

In classical group representation theory, it is well known that for a finite group inclusion $H \leqslant G$, any $H$-projective $kG$-module $V$ is a direct summand of the $kG$-module $V \downarrow_H^G \uparrow_H^G$ where $\downarrow_H^G$ is the restriction functor and $\uparrow_H^G$ is the induction functor; see \cite[Proposition 11.3.4]{ACourseInFiniteGrpRepTheory}. We have the following similar result.

\begin{lemma} \label{lemma - any_CmG-module_is_a_direct_summand_of_some_induced_module}
Let $V$ be an $\CmG$-module. Then $V$ is isomorphic to a direct summand of the $\CmG$-module $W = \Ind(\Res(V))$.
\end{lemma}
\begin{proof}
    For brevity, we write $g \cdot v$ for $(id_\bfn, g) \cdot v$ where $\bfn \in \obj(\CmG)$, $v \in V(\bfn)$, and $g \in G$. We prove by constructing a pair of $\CmG$-module homomorphisms $\varphi: V \to W$ and $\epsilon: W \to V$ such that $\epsilon \varphi = Id_V$. For an object $\bfn \in \obj(\CmG)$, the components of $\varphi$ and $\epsilon$ on $\bfn$ are given by
    \[
        \varphi_\bfn : V(\bfn) \to \Res(V)(\bfn) \otimes_k kG, \; v \mapsto \frac{1}{|G|} \sum_{g \in G} g^{-1} \cdot v \otimes_k g
    \]
    for $v \in V(\bfn)$ and
    \[
        \epsilon_\bfn : \Res(V)(\bfn) \otimes_k kG \to V(\bfn), \; u \otimes_k g \mapsto g \cdot u
    \]
    for $u \in \Res(V)(\bfn)$ and $g \in G$. It remains to check that $\epsilon_\bfn \varphi_\bfn = 1$ and that both $\varphi$ and $\epsilon$ are $\CmG$-module homomorphisms, which are routine.
\end{proof}

The functor $\Ind$ preserves injective modules. That is:

\begin{lemma} \label{lemma - Ind_preserves_injectives}
Suppose that $I$ is an injective $\Cm$-module. Then $\Ind(I) = I \boxtimes kG$ is an injective $\CmG$-module.
\end{lemma}
\begin{proof}
    Let $V$ be any $\CmG$-module. We have that
    \begin{align*}
        \Ext_{\CmG}(V,\Ind(I)) &\subseteq \Ext_{\CmG}(\Ind(\Res(V)), \Ind(I)) \\
        &= \Ext_{\Cm}(\Res(V), \Res(I \boxtimes kG)) \\
        &\cong \Ext_{\Cm}(\Res(V), \bigoplus_{g\in G} I) \\
        &\cong \bigoplus_{g\in G} \Ext_{\Cm}(\Res(V), I) \\
        &= 0.
    \end{align*}
    where the first inclusion follows from Lemma \ref{lemma - any_CmG-module_is_a_direct_summand_of_some_induced_module} together with the fact that the functor $\Ext$ is additive and the first identity follows from the Eckmann Shapiro's Lemma. Therefore, the $\CmG$-module $\Ind(I)$ is injective.
\end{proof}

An immediate corollary is:

\begin{corollary} \label{corollary - module_in_U^S_is_injective}
Every module in the class $\mathcal U^S$ is a finitely generated injective $\CSG$-module.
\end{corollary}
\begin{proof}
    This follows immediately by Lemma \ref{lemma - Ind_preserves_injectives} and \cite[Theorem 1.1]{Zeng}.
\end{proof}

In the following lemma we prove the main result of this section for the special case that $m = 1$.

\begin{lemma} \label{lemma - classification_of_inejctive_modules_of_1-rank}
Every finitely generated $\C_G$-module $V$ is finitely cogenerated by $\U^1$.
\end{lemma}
\begin{proof}
    There is a short exact sequence
    \[
        0 \to V_T \to V \to V_F \to 0
    \]
    where $V_T$ is the torsion part and $V_F$ is the torsion free part of $V$. By Lemma \ref{lemma - locally_Noetherian_property_of_CmG}, the torsion part $V_T$ is finitely generated, and hence finite dimensional. Therefore $V_T$ can be embedded into a finite dimensional injective $\C_G$-modules which lies in $\U^1$. The module $V_F$ is $[1]$-torsion free, so by Proposition \ref{proposition - S-torsion_free_module_can_be_embedded_into_S-relative_projective_module} it can be embedded into some finitely generated $[1]$-semi-induced $\C_G$-module which is projective by the comment following Definition \ref{definition - S-relative projective}. But any finitely generated projective $\C_G$-module is a direct summand of a finite direct sum of $\C_G$-modules of the form $M(n) \boxtimes kG$ which lies in $\U^1$. The conclusion then follows.
\end{proof}

\begin{lemma} \label{lemma - extension_of_scalar_on_certain_yeilds_neat_external_tensor_product}
    Notation as before. Let $\bfs$ be an object in $\CS$ and $W$ an $\Rs$-module. If $W \cong W' \boxtimes k\Aut(\bfs)$ for some $\CTG$-module $W'$, then $\Fs(W) \cong M(\bfs) \boxtimes W'$ as $\CmG$-modules. In other words, there is an isomorphism of $\CmG$-modules
    \[
        \theta: \overline M(\bfs) \otimes_{k (\mathscr{F}^{\neg S} \times G \times \mathrm{Aut}(\bfs))} W \cong M(\bfs) \boxtimes W'
    \]
    where $M(\bfs)$ is a free $\CS$-module and $\overline M(\bfs) = k\CT \boxtimes kG \boxtimes M(\bfs)$ is an $(\CmG,\Rs)$-bimodule.
\end{lemma}
\begin{proof}
    For an object $\bfn = \bfn_S \times \bfn_T \in \obj(\Cm)$ where $\bfn_S \in \obj(\CS)$ with $\bfn_S \succcurlyeq \bfs$ and $\bfn_T \in \obj(\CT)$, the map $\theta_\bfn$, which is the component of $\theta$ on $\bfn$, is given by
    \[
        \theta_\bfn: (\alpha \otimes_k g \otimes_k \beta) \otimes_{k (\mathscr{F}^{\neg S} \times G \times \mathrm{Aut}(\bfs))} (w' \otimes_k \sigma) \mapsto (\alpha,g) \cdot w' \otimes_k \beta \sigma,
    \]
    where $\alpha$ is a morphism in $\mathscr{F}^{\neg S}$ with codomain $\bfn_T$, $g \in G$, $\beta \in \CS(\bfs,\bfn_S)$, $w' \in W'$, and $\sigma \in \Aut(\bfs)$. Since
    \[
        (\alpha \otimes_k g \otimes_k \beta) \otimes_{k (\mathscr{F}^{\neg S} \times G \times \mathrm{Aut}(\bfs))} (w' \otimes_k \sigma) = (id_{\bfn_T} \otimes_k 1_G \otimes_k \beta\sigma) \otimes_{k (\mathscr{F}^{\neg S} \times G \times \mathrm{Aut}(\bfs))} ((\alpha,g) \cdot w' \otimes_k id_\bfs)
    \]
    with $\beta\sigma \in \CS(\bfs,\bfn_S)$ and $(\alpha,g) \cdot w' \in W'(\bfn_T)$, the map $\theta_\bfn$ can be simplified as
    \[
        \theta_\bfn: (id_{\bfn_T} \otimes_k 1_G \otimes_k \beta) \otimes_{k (\mathscr{F}^{\neg S} \times G \times \mathrm{Aut}(\bfs))} (w' \otimes_k id_\bfs) \mapsto \beta \otimes_k w'
    \]
    for $\beta \in \CS(\bfs,\bfn_S)$ and $w' \in W'(\bfn_T)$.
    The map $\theta_\bfn$ is easily seen to be bijective. It is routine to check that $\theta$ is an $\CmG$-module homomorphism.
\end{proof}

The following result shows that the category $\CmG \module$ has enough injectives, and furthermore gives a classification of finitely generated injective $\CmG$-modules.

\begin{theorem} \label{theorem - finitely_generated_module_is_finitely_cogenerated_by_the_class_U^m}
Any finitely generated $\CmG$-module is finitely cogenerated by $\U^m_G$.
\end{theorem}
\begin{proof}
We prove by an induction on $m$. By Lemma \ref{lemma - classification_of_inejctive_modules_of_1-rank}, the conclusion holds for $m = 1$. Suppose that the conclusion holds for all $n$ with $1 \leqslant n < m$, and let $V$ be a finitely generated $\CmG$-module. Then $V$ admits a finite filtration
\[
0 \subseteq V^{[m]}_{tor} \subseteq \ldots \subseteq V^{[i]}_{tor} \subseteq \ldots \subseteq V^{[1]}_{tor} \subseteq  V.
\]
By Corollary \ref{corollary - module_in_U^S_is_injective} and the horseshoe Lemma, it suffices to show the conclusion for each quotient $V^{[i-1]}_{tor} / V^{[i]}_{tor}$ and $V^{[m]}_{tor}$. But $V^{[m]}_{tor}$ is finite dimensional, so can be embedded into a certain finite dimensional injective module, which lies in $\U^m_G$. Therefore, we only need to show the conclusion for the quotient modules $V^{[i-1]}_{tor} / V^{[i]}_{tor}$, which is $\{i\}$-torsion free. Instead, we prove a stronger result; that is, we prove the conclusion for any nonempty subset $S$ of $[m]$ and any \Stf~$\CmG$-modules. By Corollary \ref{proposition - S-torsion_free_module_can_be_embedded_into_S-relative_projective_module}, $S$-torsion free modules can be embedded into $S$-semi-induced modules, so it suffices to show that $S$-semi-induced modules is finitely cogenerated by $\U^m_G$. By Definition \ref{definition - S-relative projective}, it turns out to show the conclusion for \Sinduced~$\CmG$-modules.

Let $V$ be a finitely generated \Sinduced~$\CmG$-module. Then it is isomorphic to $\Fs(W)$ for some finitely generated $\Rs$-module $W$ and $\bfs \in \obj(\CS)$. Put $G' = G \times \Aut(\bfs)$. Note that $\Rs = \CT \times G'$ and $|\neg S| < m$. Then by the induction hypothesis, the module $W$ is finitely cogenerated by $\U^{\neg S}_{G'}$. Without loss of generality, assume that the $\Rs$-module $W$ can be embedded as below
\[
0 \to W \to E \boxtimes k\Aut(\bfs)
\]
where $E \in \U^{\neg S}_G$. By the exactness of $\Fs$, the \Sinduced~module $V \cong \Fs(W)$ can be embedded into $\Fs( E \boxtimes k\Aut(\bfs) )$. By Lemma \ref{lemma - extension_of_scalar_on_certain_yeilds_neat_external_tensor_product}, $\Fs( E \boxtimes k\Aut(\bfs) ) \cong M(\bfs) \boxtimes E$ which lies in $\U^m_G$. This finishes the proof.
\end{proof}

From now on we focus on the category $\Cm$. By the above theorem, the class $\U^m_1$ of injective $\Cm$-modules finitely cogenerates the category $\Cm \module$, where $1$ is the trivial group. Consequently, any finitely generated injective $\Cm$-module is a direct summand of a finite direct sum of modules in $\U^m_1$. In the rest of this paper we give an explicit description of indecomposable injective $\Cm$-modules. It turns out that they coincide with external tensor products of indecomposable injective $\C$-modules, which are either indecomposable projective $\C$-modules or finite dimensional indecomposable injective $\C$-modules by \cite{SamSnowden} or \cite{GanLi}.

\begin{lemma} \label{lemma - externel_tensor_of_indinjs_is_indinj_having_local_end_ring}
Let $I_1, \ldots, I_m$ be indecomposable injective $\C$-module. Then $I_1  \boxtimes \ldots \boxtimes I_m$ is an indecomposable injective $\Cm$-module and admits a local endomorphism ring.
\end{lemma}
\begin{proof}
We only show the conclusion for $m=2$ since the general case can be proved similarly. Since $I_1 \boxtimes I_2$ is injective by \cite[Theorem 1.1]{Zeng}, it remains to show that it is indecomposable. By the classification of indecomposable $\C$-modules in \cite{SamSnowden} or \cite{GanLi}, we obtain three cases:
\begin{enumerate}
\item both $I_1$ and $I_2$ are indecomposable projective $\C$-modules;
\item both $I_1$ and $I_2$ are indecomposable finite dimensional injective $\C$-modules;
\item one of $I_1$ and $I_2$ is an indecomposable projective $\C$-module, and the other one is an indecomposable finite dimensional injective $\C$-module.
\end{enumerate}
Furthermore, $I_i$ is an indecomposable finite dimensional injective module if and only if $DI_i$ is an indecomposable projective $\C^{op}$-module, where $D = \Hom_k(-, k)$ is the usual dual functor and $\C^{op}$ is the opposite category of $\C$.

In case (1), the $\C^2$-module $I_1 \boxtimes I_2$ is an indecomposable projective $\C^2$-module, so the conclusion holds. In case (2), $D(I_1)$ and $D(I_2)$ are indecomposable projective $\C^{op}$-modules, so $D(I_1 \boxtimes I_2) \cong D(I_1) \boxtimes D(I_2)$ is an indecomposable projective $(\C^2)^{op}$-module. Consequently, the $\C^2$-module $I_1 \boxtimes I_2$ is also indecomposable. Furthermore, it is easy to see that the endomorphism ring of $I_1 \boxtimes I_2$ is local in both cases by \cite[Lemma 25.4]{Anderson}.

Now we focus on case (3). Without loss of generality we can assume that $I_1$ is an indecomposable projective $\C$-module and $I_2$ is an indecomposable finite dimensional injective $\C$-module. We want to show that the endomorphism ring of $I_1 \boxtimes I_2$ is local. Suppose that $I_1$ is induced from an irreducible left $k\Aut([n])$-module $U$ (that is, $I_1 \cong k\C \otimes_{k\Aut([n])} U$) and the $\Cop$-module $D(I_2)$ is induced from an irreducible right $k\Aut([l])$-module $W$. Since $I_1 \boxtimes I_2 \cong \Fs (U \boxtimes I_2)$ where $\bfs = [n]$, by the universal property of \Sinduced~module, we obtain a ring isomorphism $\End_{\C^2}(I_1 \boxtimes I_2) \cong \End_{\Rs}(U \boxtimes I_2)$. By Lemma \ref{lemma - Ind_preserves_injectives}, the $\Rs$-module $U \boxtimes I_2$ is injective. Moreover, it is indecomposable since the dual $\Rs^{op}$-module $D(U \boxtimes I_2) \cong D(U) \boxtimes D(I_2)$ is induced from the irreducible right $k(\Aut([n]) \times \Aut([l]))$-module $D(U) \boxtimes W$. Therefore by \cite[Lemma 25.4]{Anderson}, the endomorphism ring $\End_{\Rs}(U \boxtimes I_2)$ is local as desired.
\end{proof}

As a corollary, we have:

\begin{corollary} \label{corollary - ind_direc_summand_of_sum_of_exter_prod_of_ind_injs}
Let $E_i = I^i_1 \boxtimes \ldots \boxtimes I^i_m$ be an $\Cm$-module for $i \in [n]$, where each $I^i_j$ is an indecomposable injective $\C$-module for $j \in [m]$. Then any indecomposable direct summand of the $\Cm$-module $\oplus_{i=1}^n E_i$ is isomorphic to certain $E_i$.
\end{corollary}

\begin{proof}
    This follows from Lemma \ref{lemma - externel_tensor_of_indinjs_is_indinj_having_local_end_ring} and \cite[Theorem 12.6(2)]{Anderson}.
\end{proof}

Now we are ready to prove the second theorem in the Introduction of this paper.

\begin{proof}[A proof of Theorem \ref{theorem - classification_of_f.g._injective_FIm-module}]
    By Theorem \ref{theorem - finitely_generated_module_is_finitely_cogenerated_by_the_class_U^m} and Corollary \ref{corollary - module_in_U^S_is_injective}, any indecomposable injective $\Cm$-module $I$ is isomorphic to an indecomposable direct summand of a finite direct sum of modules in $\U^m_1$. Since the external tensor product $\boxtimes$ commutes with the direct sum $\oplus$, such a finite direct sum can be written as $\oplus_{i=1}^n E_i$ where $E_i$ is described in Corollary \ref{corollary - ind_direc_summand_of_sum_of_exter_prod_of_ind_injs}. Thus, its indecomposable direct summand is isomorphic to a certain $E_i$, as desired.
\end{proof}

\end{document}